\documentclass[oneside,reqno,12pt]{amsart}

\usepackage{amsmath,amssymb,amsfonts}
\usepackage{kpfonts}
\usepackage{xcolor}
\usepackage{mathtools}
\usepackage{bbm}
\usepackage{hyperref}
\usepackage[foot]{amsaddr}

\newtheorem{theorem}{Theorem}
\newtheorem{lemma}[theorem]{Lemma}
\newtheorem{corollary}[theorem]{Corollary}
\newtheorem{rem}[theorem]{Remark}
\newtheorem{prop}[theorem]{Proposition}

\newcommand{\1}{\mathbbm{1}}

\renewcommand{\mod}{{\rm mod}\,}
\renewcommand{\dim}{{\rm dim}\,}

\newcommand{\R}{\mathbb{R}}

\newcommand{\N}{\mathbb{N}}
\newcommand{\Q}{\mathbb{Q}}
\newcommand{\E}{\mathbb{E}}
\newcommand{\Z}{\mathbb{Z}}

\newcommand{\GCD}{{\rm GCD}\,}
\newcommand{\LCM}{{\rm LCM}\,}
\newcommand{\NLCM}{{\rm NLCM}\,}
\renewcommand{\dim}{{\rm dim}\,}
\newcommand{\ukp}{u_{k,p}}

\newcommand{\todistr}{\overset{{\rm d}}{\longrightarrow}}
\newcommand{\toprobab}{\overset{\mathbb{P}}{\longrightarrow}}

\numberwithin{equation}{section}
\numberwithin{theorem}{section}

\begin{document}

\title[Polynomials of random integers]{Divisibility properties of polynomial expressions of random integers}

\author{Zakhar Kabluchko$^1$}
\address{$^1$Institut f\"{u}r Mathematische Stochastik, Westf\"{a}lische Wilhelms-Universit\"{a}t M\"{u}nster, M\"{u}nster, Germany}

\author{Alexander Marynych$^2$}
\address{$^2$Faculty of Computer Science and Cybernetics, Taras Shev\-chen\-ko National University of Kyiv, Kyiv 01601, Ukraine}

\begin{abstract}
We study divisibility properties of a set $\{f_1(\mathbf{U}_n^{(s)}),\ldots,f_m(\mathbf{U}_n^{(s)})\}$, where $f_1,\ldots,f_m$ are polynomials in $s$ variables over $\mathbb{Z}$ and $\mathbf{U}_n^{(s)}$ is a point picked uniformly at random from the set $\{1,\ldots,n\}^s$, $s\in\mathbb{N}$. We show that the $\GCD$ and the suitably normalized $\LCM$ of this set converge in distribution to a.s.\ finite random variables under mild assumptions on $f_1,\ldots, f_m$. Our approach is based on the notion of integer adeles and a known fact that the uniform distribution on $\{1,\ldots, n\}$ converges to the Haar measure on the ring of integer adeles combined with the Lang-Weil bounds.
\end{abstract}



\maketitle

\section{Introduction}

One of the most classic results in probabilistic number theory, which can be traced back at least to Dirichlet~\cite{Dirichlet}, states that probability that two numbers, picked uniformly at random from the set $\{1,2,\ldots,n\}$, are coprime converges to
$$
\prod_{p\in\mathcal{P}}\left(1-\frac{1}{p^2}\right)=\frac{1}{\zeta(2)}=\frac{6}{\pi^2},
$$
as $n\to\infty$. Here $\mathcal{P}$ denotes the set of prime numbers and $\zeta$ is the Riemann zeta-function. We refer to~\cite{Abramovich+Nikitin} for a nice historical account of this result. More generally, it is known~\cite{Cesaro,Christopher,Cohen} that the greatest common divisor, to be denoted in what follows by $\GCD$, of $s\geq 2$ numbers $(U_{n,1},\ldots,U_{n,s})=:{\bf U}_n^{(s)}$ picked uniformly at random from $\{1,2,\ldots,n\}^s$ converges in distribution, as $n\to\infty$, to an $\N$-valued random variable with the probability mass function
\begin{equation}\label{eq:zeta_distribution}
j\longmapsto\frac{1}{\zeta(s)j^s},\quad j\in\mathbb{N}.
\end{equation}
A recent paper~\cite{Fernandez+Fernandez} provides a comprehensive overview of results of this kind related to divisibility of random integers.

The motivation for the present paper comes from our attempt to understand the aforementioned result via a continuous mapping approach ubiquitous in probability theory and also to generalize it. In its simplest form, the continuous mapping theorem, see Theorem 2.7 in~\cite{Billingsley}, states that if a sequence of random elements $(X_n)_{n\in\mathbb{N}}$ with values in a metric space $M_1$ converges in distribution to a random element $X_{\infty}$ and $f$ is a continuous mapping from $M_1$ to another metric space $M_2$, then a sequence of $M_2$-valued random elements $(f(X_n))_{n\in\mathbb{N}}$ converges in distribution to $f(X_{\infty})$. Thus, to derive a convergence of $\GCD({\bf U}_n^{(s)})$ via the continuous mapping approach the crucial step is to pick an appropriate topology, with respect to which the convergence in distribution of ${\bf U}_n^{(s)}$ is regarded. In this respect, the notion of integer adeles and a closely related concept of profinite integers turned out to be very useful. An incomplete list of references on various applications of profinite integers and integer adeles in probabilistic number theory includes~\cite{Avdeeva+Cellarosi+Sinai,Duy1,Duy+Takanobu,Kubota+Sugita, Novoselov1,Novoselov2,Sugita+Takanobu}.

Roughly speaking, a ring of integer adeles $\widehat{\mathbb{Z}}$ is a compactification of $\mathbb{Z}$ with respect to which two integers are close, if and only if they possess the same small prime divisors counting multiplicities. The rigorous definition will be recalled below in Section~\ref{sec:results_adeles}. A nice overview of this and other compactifications of $\Z$ in probabilistic number theory can be found in~\cite{Indlekofer1,Indlekofer2}. In particular, a proof of the aforementioned result on the density of coprime pairs using this notion was given in~\cite{Kubota+Sugita}; see also~\cite{Duy1,Duy+Takanobu,Sugita+Takanobu} for related results. A rather simple observation which lies in the core of those proofs is the convergence of $U_n:=U_{n,1}$ (and, therefore, of ${\bf U}_n^{(s)}$), identified with an element of $\widehat{\mathbb{Z}}$ via the canonical embedding, to a random element distributed according to the Haar measure on $\widehat{\mathbb{Z}}$. This result can be found, for example, as Lemma 6 in~\cite{Kubota+Sugita}. We shall recall this fact in a slightly extended form as Proposition~\ref{prop:main_adeles} and give a short proof based on the Chinese Remainder Theorem in Section~\ref{sec:results_adeles}.

In this paper we are concerned with generalizations of the aforementioned and similar results to the following more general setting. Let $f_1,\ldots,f_m$ be polynomials in $s$ variables with integer coefficients. What can be said about the greatest common divisor of $\{f_1(\mathbf{U}_n^{(s)}),\ldots,f_m(\mathbf{U}_n^{(s)})\}$? Or what is the probability that $f_1(\mathbf{U}_n^{(s)})$ divides $f_2(\mathbf{U}_n^{(s)})$? It turn out that the approach outlined above can be successfully applied in this setting. For example, our results can be used to conclude that the sequence of random variables
\begin{equation}\label{eq:lcm_intro}
n^{-9}\LCM(U_{n,1}^2+U_{n,2}^2,U_{n,1}^3+U_{n,2}^3,U_{n,1}^4+U_{n,2}^4),
\end{equation}
where $\LCM$ denotes the least common multiple, converges in distribution to a non-trivial limit, as $n\to\infty$. A direct check of this fact seems to be a challenging problem. To the best of our knowledge, these questions have not been addressed in the literature. A related result on relatively prime values of polynomials can be found in Theorem 3.1 in~\cite{Poonen}. Another tangent result is an {E}rd\H{o}s-{K}ac law for the number of prime divisors of a polynomial in several variables which has been established in~\cite{Xiong}.

The rest of the paper is organized as follows. In Section~\ref{sec:results_adeles} we recall the definition of integer adeles and reprove a result on convergence of all $p$-adic expansions of a uniformly sampled integer on $\{1,2,\ldots,n\}$ to a random element of $\widehat{\Z}$ distributed according the Haar measure. Section~\ref{sec:results_polynomials} is devoted to the the analysis of arithmetic properties of the set $\{f_1(\mathbf{U}_n^{(s)}),\ldots,f_m(\mathbf{U}_n^{(s)})\}$.
One of the central result in Section~\ref{sec:results_polynomials} is Theorem~\ref{theo:gcd} which, in particular, provides the limit distribution for the $\LCM$ in~\eqref{eq:lcm_intro}. A list of further results in Section~\ref{sec:results_polynomials} includes limit theorems for the $\GCD$ and the normalized $\LCM$ of the above set. The proofs of these results are given in Section~\ref{sec:proofs} with a one long technical proof being postponed to Section~\ref{sec:abscence}. Some short auxiluary results are collected in the Appendix.

\section{Ring of integer adeles and convergence to the Haar measure}\label{sec:results_adeles}

Let $\Q_p$ be the field of $p$-adic rational numbers, which is the completion of $\Q$ with respect to the $p$-adic norm
$$
\left\|\frac{a}{b}p^l\right\|_p:=p^{-l},\quad l\in\Z,\quad a,b\text{ are coprime to }p.
$$
Denote also by $\|\cdot\|_{\infty}$ the usual Euclidean norm on $\Q$ and by $\Q_{\infty}=\R$ the completion of $\Q$ with respect to $\|\cdot\|_{\infty}$.

For $p\in\mathcal{P}$ let $\Z_p$ be the ring of $p$-adic integers in $\Q_p$.
Any $p$-adic integer can be represented as  $a_0+a_1 p+a_2 p^2+\cdots$ with $a_i\in\{0,1,\ldots,p-1\}=:\Z/p\Z$. The ring $\Z_p$ is a compact subring of $\Q_p$. Therefore, the direct product
$$
\widehat{\mathbb Z}=\prod_{p\in\mathcal{P}}\mathbb{Z}_p,
$$
is also a compact topological ring by Tychonoff's theorem.  The elements of $\widehat{\mathbb Z}$ are called  {\it integer adeles}~\cite{Lang},  \textit{profinite integers}~\cite{Rota} or \textit{polyadic numbers}~\cite{Novoselov1,Novoselov2}.  The compact abelian group $\widehat{\mathbb Z}$ is the profinite completion of $\Z$ introduced by H.\ Pr\"ufer~\cite{Pruefer}; see also~\cite{Demangos+Longhi,Lovas+Mezoe,Novoselov1,Novoselov2}.

Since $\Z_p$ is a compact abelian group, for each $p\in\mathcal{P}$, there exists a unique invariant (Haar) probability measure $\mu_p$ on $\Z_p$. The explicit probabilistic construction of $\mu_p$ is as follows. Take $(\ukp)_{k\geq 0}$ independent uniformly distributed on $\{0,1,\ldots,p-1\}$ random variables and put
\begin{equation}\label{eq:v_p}
V_p:=\sum_{k=0}^{\infty}\ukp p^k\in\Z_p.
\end{equation}
Then $\mu_p$ is the distribution of $V_p$. Let $\widehat{\mu}=\prod_{p\in\mathcal{P}}\mu_p$ be the product measure on $\widehat{\mathbb Z}$. Then, $\widehat {\mu}$ is the unique Haar probability measure on the compact group $\widehat {\Z}$.

There is a unique canonical ring homomorphism
$$
\phi:\mathbb{Z}\longrightarrow\widehat{\mathbb Z}
$$
with $\phi(1) = 1$. It sends an integer $n$ to an infinite vector $\phi(n):=(\phi_p(n))_{p\in\mathcal{P}}\in \widehat{\mathbb Z}$ such that $\phi_p(n)$ is the $p$-adic expansion of $n$. Let $\pi^{(p)}:\widehat{\mathbb Z}\to \Z_p$ and $\pi^{(p)}_j:\widehat{\mathbb Z}\to\{0,1,\ldots,p-1\}$ be the canonical projections
\begin{equation}\label{eq:projections}
\pi^{(p)}((x_p)_{p\in \mathcal{P}})=x_p\quad\text{and}\quad \pi^{(p)}_j\left(\left(\sum_{k=0}^{\infty}a_{k,p} p^k\right)_{p\in\mathcal{P}}\right)=a_{j,p},\quad j\geq 0,\quad p\in\mathcal{P}.
\end{equation}

We shall use $\todistr$ to denote convergence in distribution of random elements. Throughout the paper convergence of infinite-dimensional vectors is understood with respect to the product topology, that is, as convergence of all finite-dimensional projections. A version of Proposition~\ref{prop:main_adeles} can be found as Lemma 6 in~\cite{Kubota+Sugita}.

\begin{prop}\label{prop:main_adeles}
Let $U_n$ be a random variable with the uniform distribution on $\{1,2,\ldots,n\}$. Then we have the convergence in distribution
$$
\left(\phi(U_n),\frac{U_n}{n}\right)~\todistr~(\mathcal{V},U_{\infty}),\quad n\to\infty,
$$
on the space $\widehat{\Z} \times [0,1]$. Here $\mathcal{V}:=(V_p)_{p\in\mathcal{P}}$, $V_p$  is given by~\eqref{eq:v_p}, $U_{\infty}$ has the uniform distribution on $[0,1]$, and $U_{\infty},V_2,V_3,V_5,\ldots$ are mutually independent. Note that $(\mathcal{V},U_{\infty})$ is distributed according to the product of the Haar measure  $\widehat{\mu}$ on $\widehat {\Z}$ and the Lebesgue measure on $[0,1]$.
\end{prop}
\begin{proof}
We need to show that
$$
\left(\left(\pi^{(p)}_j(\phi(U_n))\right)_{p\in\mathcal{P},j\in\mathbb{N}},\frac{U_n}{n}\right)~\todistr~\left(\left(\pi^{(p)}_j(\mathcal{V})\right)_{p\in\mathcal{P},j\in\mathbb{N}},U_{\infty}\right),\quad n\to\infty.
$$
Fix pairwise distinct $p_1,p_2,\ldots,p_m\in\mathcal{P}$, arbitrary $l_1,l_2,\ldots,l_m\in\N$, $t\in [0,1]$ and note that by independence
$$
\mathbb{P}\{\pi^{(p_i)}_k(\mathcal{V})=r_{k,p_i},\quad i=1,\ldots,m,\quad k=0,\dots,l_i-1\}=\prod_{i=1}^{m}\frac{1}{p_i^{l_i}},
$$
for any $r_{k,p_i}\in\Z/p_i\Z$, $i=1,\ldots,m$. Thus, it suffices to show that
\begin{equation}\label{eq:thm1_proof}
\lim_{n\to\infty}\mathbb{P}\{\pi^{(p_i)}_k(\phi(U_n))=r_{k,p_i},\quad i=1,\ldots,m,\quad k=0,\dots,l_i-1,\quad U_n\leq nt\}=t\prod_{i=1}^{m}\frac{1}{p_i^{l_i}},
\end{equation}
for any $r_{k,p_i}\in\Z/p_i\Z$, $i=1,\ldots,m$. Put $r_i:=\sum_{j=0}^{l_i-1}r_{j,p_i}p_i^j$ and note that
\begin{align*}
&\hspace{-1cm}\mathbb{P}\{\pi^{(p_i)}_k(\phi(U_n))=r_{k,p_i},\quad i=1,\ldots,m,\quad k=0,\dots,l_i-1,\quad U_n\leq nt\}\\
&=\mathbb{P}\{U_n\equiv r_{i}(\mod p^{l_i}),\quad i=1,\ldots,m,\quad U_n\leq nt\}\\
&=\frac{1}{n}\#\left\{k\in \{1,2,\ldots,\lfloor nt\rfloor\}: k\equiv r_i\left(\mod p_i^{l_i}\right),i=1,\ldots,m\right\}.
\end{align*}
Put $M:=\prod_{i=1}^{m}p_i^{l_i}$ and let $\1\{A\}$ denote the indicator of the event $A$. By the Chinese remainder theorem, for some unique $r\in\Z/M\Z$,
\begin{multline*}
\frac{1}{n}\#\left\{k\in \{1,2,\ldots,\lfloor nt\rfloor\}: k\equiv r_i\left(\mod p_i^{l_i}\right),i=1,\ldots,m\right\}\\
=\frac{1}{n}\#\left\{k\in \{1,2,\ldots,\lfloor nt\rfloor\}: k\equiv r\left(\mod M\right)\right\}=\frac{1}{n}\sum_{l\geq 0}\1\{r+lM\leq nt\}=\frac{1}{n}\left\lfloor\frac{nt-r}{M}\right\rfloor,
\end{multline*}
and the right-hand side converges to $tM^{-1}$, as $n\to\infty$. Thus,~\eqref{eq:thm1_proof} holds and the proof is complete.
\end{proof}

Let $f_1,\ldots,f_m\in \Z[x_1,\ldots,x_s]$ be $m$ polynomials in $s$ variables over $\Z$. Since $\phi$ is a homomorphism, we have $\phi(f(x_1,\ldots,x_s))=f(\phi(x_1),\ldots,\phi(x_s))$, for every $f\in\Z[x_1,\ldots,x_s]$. By the continuous mapping theorem we obtain the following corollary. Let $\mathcal{V}_1,\ldots,\mathcal{V}_s$ be independent copies of $\mathcal{V}$ and recall the notation $\mathbf{U}_n^{(s)} = (U_{n,1},\ldots,U_{n,s})$ for a uniformly distributed on $\{1,\ldots, n\}^s$ random vector. Then, with $\mathbf{U}^{(s)}_{\infty}$ being uniformly distributed on $[0,1]^s$ and independent of $\mathcal{V}_1,\ldots,\mathcal{V}_s$ the following corollary holds true.

\begin{corollary}\label{cor:poly}
As $n\to\infty$,
$$
\left(\phi(f_1(\mathbf{U}_n^{(s)})),\ldots,\phi(f_m(\mathbf{U}_n^{(s)})),\frac 1n \mathbf{U}_n^{(s)}\right)~\todistr~\left(f_1(\mathcal{V}_1,\ldots,\mathcal{V}_s),\ldots,f_m(\mathcal{V}_1,\ldots,\mathcal{V}_s),\mathbf{U}^{(s)}_{\infty}\right).
$$
\end{corollary}
In what follows it is important that, for every fixed $f\in \Z[x_1,\ldots,x_s]$, the projections $\pi^{(p)}(f(\mathcal{V}_1,\ldots,\mathcal{V}_s))$, $p\in\mathcal{P}$, are mutually independent. This follows from the fact that $\pi^{(p)}(\mathcal{V})$, $p\in\mathcal{P}$, are independent and $\pi^{(p)}$ is a ring homomorphism, thus, commutes with any polynomial.

For $n\in\Z\setminus\{0\}$ and $p\in\mathcal{P}$ let $\lambda_p(n)$ denote the power of prime $p$ in the prime decomposition of $|n|$, so
$$
|n|=\prod_{p\in\mathcal{P}}p^{\lambda_p(n)}.
$$
We have an obvious relation $\lambda_p(n)=\inf\{k\geq 0:\pi_k^{(p)}(\phi(n))>0\}$,
which advocates the usage of the same notation $\lambda_p$ for the following function defined on $\widehat{\mathbb Z}$:
$$
\lambda_p\left(x\right)=\inf\{k\geq 0:\pi_k^{(p)}(x)>0\},\quad x\in\widehat{\mathbb Z}.
$$
This definition also shows that it is natural to stipulate $\lambda_p(0):=+\infty$, $p\in\mathcal{P}$. Our main result implies the next two corollaries. The first one is well-known, see, for instance, Lemma 3.1 in~\cite{BosMarRas:2019}, the second one seems to be new. Set
$$
\mathcal{G}_p:=\lambda_p(\mathcal{V})=\inf\{k\geq 0:\pi_k^{(p)}(\mathcal{V})>0\},\quad p\in\mathcal{P}.
$$
\begin{corollary}\label{cor:geometrics}
The random variables $\mathcal{G}_p$, $p\in\mathcal{P}$, and $U_{\infty}$ are mutually independent and $\mathcal{G}_p$ has a geometric distribution
\begin{equation}\label{eq:geometric}
\mathbb{P}\{\mathcal{G}_p\geq k\}=\frac{1}{p^k},\quad k\geq 0,\quad p\in\mathcal{P}.
\end{equation}
Furthermore,
$$
\left(\left(\lambda_p(U_n)\right)_{p\in\mathcal{P}},\frac{U_n}{n}\right)~\todistr ((\mathcal{G}_p)_{p\in\mathcal{P}},U_{\infty}),\quad n\to\infty.
$$
\end{corollary}

\begin{corollary}\label{cor:polys_p_adic_norms_converge}
For polynomials $f_1,\ldots, f_m\in \Z[x_1,\ldots,x_s]$ with the same notation as in Corollary~\ref{cor:poly} we have
$$
\left(\lambda_p(f_i(\mathbf{U}^{(s)}_n))\right)_{p\in\mathcal{P}, i=1,\ldots, m}~\todistr~\left(\lambda_p(f_i(\mathcal{V}_1,\ldots,\mathcal{V}_s))\right)_{p\in\mathcal{P}, i=1,\ldots, m},\quad n\to\infty.
$$
For every $i\in \{1,\ldots, m\}$, the limiting random variables $\lambda_p(f_i(\mathcal{V}_1,\ldots,\mathcal{V}_s))$, $p\in\mathcal{P}$, are mutually independent.
\end{corollary}
\begin{prop}\label{prop:lambdas_finite}
If $f\in \Z[x_1,\ldots,x_s]$ is a non-zero polynomial with integer coefficients, then $\mathbb{P}\{\lambda_p(f(\mathcal{V}_1,\ldots,\mathcal{V}_s)=+\infty\}=0$ for every $p\in\mathcal{P}$.
\end{prop}
\begin{proof}
Recall  that $(\pi^{(p)}(\mathcal{V}_1),\ldots,\pi^{(p)}(\mathcal{V}_s))$ is distributed according to the product measure $\mu_p^{\otimes s}$, where $\mu_p$ is the Haar measure on $\Z_p$. Hence, $\mathbb{P}\{\lambda_p(f(\mathcal{V}_1,\ldots,\mathcal{V}_s))=+\infty\} = \mu_p^{\otimes s}(\{x\in \mathbb{Z}_p^s:f(x)=0\})$.
By Lemma~\ref{lem:haar_of_variety_is_zero} in the Appendix the right-hand side is equal to zero.
\end{proof}

\section{Main results}\label{sec:results_polynomials}

For a multiset $A:=\{a_1,\ldots,a_m\}\subset\Z$, let $\GCD(A)$ denote the greatest common divisor, $\LCM(A)$ the least common multiple and $\NLCM(A)$ the normalized least common multiple of a multiset $\{|a_1|,\ldots,|a_m|\}\subset\{0,1,2,\ldots\}$, respectively.  If $0\notin A$, then $\NLCM(A)$ by definition is equal to
$$
\NLCM(A):=\frac{\LCM(A)}{\prod_{i=1}^{m}|a_i|}.
$$
If $A$ contains zero, then we stipulate $\GCD(A):=\GCD(A\setminus\{0\})$, $\LCM(A):=0$ and $\NLCM(A):=\NLCM(A\setminus\{0\})$.

\begin{theorem}\label{theo:gcd}
Let $f_1,\ldots, f_m\in\Z[x_1,\ldots,x_s]$ be $m\geq 2$ non-zero polynomials that do not have a common factor of degree $>0$. Then
$$
\GCD(f_1(\mathbf{U}_n^{(s)}),\ldots,f_m(\mathbf{U}_n^{(s)}) )~\todistr~G_{f_1,\ldots, f_m}, \quad n\to\infty,
$$
for some random variable $G_{f_1,\ldots, f_m}$ with values in $\N$. More concretely, we have
$$
\log G_{f_1,\ldots, f_m}=\sum_{p\in\mathcal{P}}\log p \min_{i=1,\ldots, m} \lambda_p (f_i(\mathcal{V}_1,\ldots,\mathcal{V}_s)),
$$
and the series on the right-hand side converges a.s.
\end{theorem}

The proof of Theorem~\ref{theo:gcd} will be given in Section~\ref{subsec:gcd}.

\begin{rem}\label{rem:joint_convergence_all_systems_of_polys}
In fact, almost the same proof shows a slightly stronger version of Theorem~\ref{theo:gcd} in which the GCD's of all tuples of polynomials without a common factor converge \textit{jointly} in distribution. More precisely, for integer $m\geq 2$ let $\mathbb{I}_m$ be the set of all $m$-tuples $(f_1,\ldots, f_m)$ of polynomials from $\Z [x_1,\ldots, x_s]$ that do not have a common factor of degree $>0$. Then, the following distributional convergence of collections of random variables holds:
$$
(\GCD(f_1(\mathbf{U}_n^{(s)}),\ldots,f_m(\mathbf{U}_n^{(s)}))_{m\geq 2, (f_1,\ldots, f_m) \in \mathbb{I}_m}
~\todistr~
(G_{f_1,\ldots, f_m})_{m\geq 2, (f_1,\ldots, f_m) \in \mathbb{I}_m},
\quad n\to\infty.
$$
\end{rem}

\begin{rem}
Upon setting $s:=m$, $f_j(x_1,\ldots,x_s):=x_j$, $j=1,\ldots,m$, we recover the result mentioned in the introduction. Namely,
$$
\GCD(U_{n,1},\ldots,U_{n,s})~\todistr~\prod_{p\in\mathcal{P}}p^{\min_{k=1,\ldots,s}\mathcal{G}_{p,k}},\quad n\to\infty,
$$
where $(\mathcal{G}_{p,k})_{p\in\mathcal{P},k=1,\ldots,s}$ are mutually independent and $\mathcal{G}_{p,k}$ has the geometric distribution~\eqref{eq:geometric} for every $k=1,\ldots,s$ and $p\in\mathcal{P}$. By calculating the moments $\E\left(\prod_{p\in\mathcal{P}}p^{-t\min_{k=1,\ldots,s}\mathcal{G}_{p,k}}\right)$, $t>0$, with the aid of Euler's product formula, we see that
$$
\mathbb{P}\left\{\prod_{p\in\mathcal{P}}p^{\min_{k=1,\ldots,s}\mathcal{G}_{p,k}}=j\right\}=\frac{1}{\zeta(s)j^s},\quad j\in\N,
$$
in full accordance with~\eqref{eq:zeta_distribution}. Interestingly, this distribution has pop up also in the context of profinite integers in~\cite{Rota}.
\end{rem}

\begin{corollary}\label{theo:divisible}
Let $f,g\in\Z[x_1,\ldots,x_s]$ be two polynomials such that $f$ does not divide $g$ over $\mathbb{Q}[x_1,\ldots, x_s]$. Assume that $\deg f \geq 1$. Then
$$
\lim_{n\to\infty}\mathbb{P}\{f(\mathbf{U}_n^{(s)})\text{ divides } g(\mathbf{U}_n^{(s)})\}=0.
$$
\end{corollary}

The proof of Corollary~\ref{theo:divisible} will be given in Section~\ref{subsec:divisible}.

\begin{theorem}\label{theo:lcm}
Let $f_1,\ldots, f_m\in\Z[x_1,\ldots,x_s]$ be $m\geq 2$ non-zero polynomials such that any pair does not share a common factor of degree $>0$. Put $d_i:=\deg f_i$. Then
\begin{equation}\label{eq:lcm_convergence}
\NLCM(f_1(\mathbf{U}_n^{(s)}),\ldots,f_m(\mathbf{U}_n^{(s)}))~\todistr~L_{f_1,\ldots, f_m}, \quad n\to\infty,
\end{equation}
for some random variable $L_{f_1,\ldots, f_m}$ with values in $1/\N:=\{1,1/2,1/3,\ldots\}$. More concretely, we have
$$
\log L_{f_1,\ldots, f_m}=\sum_{p\in\mathcal{P}}\log p \left(\max_{i=1,\ldots, m} \lambda_p (f_i(\mathcal{V}_1,\ldots,\mathcal{V}_s))-\sum_{i=1}^{m} \lambda_p (f_i(\mathcal{V}_1,\ldots,\mathcal{V}_s))\right),
$$
and the series on the right-hand side converges a.s. Moreover, with $\bar f_i \in \Z[x_1,\ldots, x_s]$ denoting a homogeneous polynomial of the same degree as $f_i$ obtained from $f_i$ by dropping all monomials except those having the highest degree $d_i$,  it holds
\begin{equation}\label{eq:lcm_convergence2}
\frac{\LCM(f_1(\mathbf{U}_n^{(s)}),\ldots,f_m(\mathbf{U}_n^{(s)}))}{n^{d_1+\cdots+d_m}}~\todistr~L_{f_1,\ldots, f_m}\prod_{i=1}^{m}\bar{f}_i(\mathbf{U}_{\infty}^{(s)}),
\end{equation}
where $\mathbf{U}_{\infty}^{(s)}$ is independent of $L_{f_1,\ldots, f_m}$ and has the uniform distribution on $[0,1]^s$.
\end{theorem}

The proof of Theorem~\ref{theo:lcm} will be given in Section~\ref{subsec:lcm}.

\section{Proofs for Section~\ref{sec:results_polynomials}}\label{sec:proofs}

\subsection{Preliminaries: algebraic sets and varieties}
Here we recall some basic notions from algebraic geometry and prove several auxiliary results needed for the proof of Theorem~\ref{theo:gcd}. We refer to Chapter VI in~\cite{Schmidt} for the definitions given below and further properties of algebraic sets and varieties.

Let $\mathbb{K}$ be a field and denote by $\overline{\mathbb{K}}$ its algebraic closure. For a subset $S$ of the ring $\mathbb{K}[x_1,\ldots,x_s]$ of polynomials over $\mathbb{K}$ the set
\begin{equation}\label{eq:variety}
A_{\mathbb{K}}(S):=\{ x \in \overline{\mathbb{K}}^s: g(x) = 0\;\forall g\in S\}
\end{equation}
is called an (affine) $\mathbb{K}$-algebraic set. A $\mathbb{K}$-algebraic set is called irreducible (or an affine $\mathbb{K}$-algebraic variety) if it is not the union of two strictly smaller $\mathbb{K}$-algebraic sets. Every $\mathbb{K}$-algebraic set is a finite union of irreducible $\mathbb{K}$-algebraic varieties, called irreducible components. This decomposition is unique, see Theorems 1I and 1J in~\cite[Chapter VI]{Schmidt}. The dimension of a $\mathbb{K}$-algebraic variety $A_{\mathbb{K}}(S)$ is the maximal length $d\in \{0,1,\ldots,s\}$ of the chains $V_0\subset V_1\subset\cdots\subset V_d$ of distinct nonempty $\mathbb{K}$-algebraic subvarieties of $A_{\mathbb{K}}(S)$. The dimension of a $\mathbb{K}$-algebraic set is the maximum of the dimensions of its irreducible components. The algebraic varieties of dimension $s-1$ are called hypersurfaces.

\begin{lemma}\label{lem:dimension}
Assume that $f_1,\ldots,f_m\in \Z[x_1,\ldots,x_s]$ are $m\geq 2$ non-zero polynomials that do not have a common factor of degree $>0$, then
$$
\dim (A_{\mathbb{Q}}(f_1,\ldots,f_m))\leq s-2.
$$
For $s=1$ this means that $A_{\mathbb{Q}}(f_1,\ldots,f_m)$ is empty.
\end{lemma}
\begin{proof}
We argue by contradiction. At least one of the polynomials is non-zero, therefore $\dim (A_{\mathbb{Q}}(f_1,\ldots,f_m))<s$. Assume that $\dim (A_{\mathbb{Q}}(f_1,\ldots,f_m))=s-1$, then at least one irreducible component of $A_{\mathbb{Q}}(f_1,\ldots,f_m)$ is a hypersurface, say $H$. According to Theorem 2C(ii) in~\cite[Chapter VI]{Schmidt}, there exists an irreducible polynomial $h\in \mathbb{Q}[x_1,\ldots,x_s]$ such that
$$
H:=A_{\mathbb{Q}}(h)=\{x \in \overline{\mathbb{Q}}^s: h(x) = 0\},\quad \deg h\geq 1.
$$
Since $f_1,\ldots,f_m$ vanish on $H$, Hilbert's Nullstellensatz yields that
$$
f_i^{r_i}=hg_i,\quad i=1,\ldots,m,
$$
for some $g_1,\ldots,g_m\in \mathbb{Q}[x_1,\ldots,x_s]$ and $r_1,\ldots,r_m\in\N$. Thus, $h$ is a common factor of $f_1,\ldots,f_m$ giving the desired contradiction.
\end{proof}

Any set of polynomials $f_1,\ldots,f_m\in \Q[x_1,\ldots,x_s]$ with rational coefficients can be regarded also as a set of polynomials over finite fields
$\mathbb{F}_p$, $p\in\mathcal{P}$, by reducing their coefficients modulo $p$.
The next result shows that basic characteristics of the $\mathbb{Q}$-algebraic set $A_{\mathbb{Q}}(f_1,\ldots,f_m)$, such as the number of irreducible components and the dimension, are preserved when passing to $\mathbb{F}_p$-algebraic sets $A_{\mathbb{F}_p}(f_1,\ldots,f_m)$, for all but finitely many primes $p\in\mathcal{P}$.

\begin{prop}\label{prop:vaiety_Q-variety_p}
Let $f_1,\ldots, f_m \in \mathbb{Q}[x_1,\ldots, x_s]$ be such that the algebraic set $A_{\mathbb{Q}}(f_1,\ldots,f_m)$ has $\ell$ irreducible components and dimension $d$. Then, for all but finitely many primes $p$, the variety $A_{\mathbb{F}_p}(f_1,\ldots,f_m)$ has $m$ irreducible components and the same dimension $d$.
\end{prop}
\begin{proof}
The claim about the number of components follows from Proposition 5 in~\cite{Greenleaf}. Thus, we may assume that $\ell=1$, that is, $A_{\mathbb{Q}}(f_1,\ldots,f_m)$ is a $\mathbb{Q}$-algebraic variety of dimension $d$. By Corollary 10.4.3 in~\cite{Fried+Jarden} the dimension of $A_{\mathbb{F}_p}(f_1,\ldots,f_m)$ is equal to $d$ for all but finitely many primes $p\in\mathcal{P}$.
\end{proof}

A complexity of a $\mathbb{K}$-algebraic set $A_{\mathbb{K}}(f_1,\ldots,f_m)$, for $f_1,\ldots,f_m\in \mathbb{K}[x_1,\ldots,x_s]$, is defined as the maximum of $s$, $m$ and the degrees of $f_1,\ldots,f_m$. Proposition~\ref{prop:vaiety_Q-variety_p} in conjunction with the classical Lang-Weil bound, see the original work~\cite{LangWeil} or~Theorem 4.1 in~\cite{Ghorpade}, yields the following.

\begin{prop}[The Lang-Weil bound]\label{prop:lang_weil}
For $f_1,\ldots, f_m \in \mathbb{Z}[x_1,\ldots, x_s]$, consider a $\mathbb{Q}$-al\-geb\-raic set $V:=A_{\mathbb{Q}}(f_1,\ldots,f_m)$ of complexity at most $M$. Then, for all but finitely many $p\in\mathcal{P}$,
$$
\#\{x\in\mathbb{F}^s_p:f_1(x)=\cdots=f_m(x)=0\}=(\ell(V)+O(p^{-1/2}))p^{\dim(V)},
$$
where $\ell(V)\in\N$ is the number of irreducible components of $V$ of dimension $\dim(V)$ and a constant in the Landau symbol $O$ depends only on the complexity $M$. In particular, if $V$ is irreducible than $\ell(V)=1$.
\end{prop}

\subsection{Proof of Theorem~\ref{theo:gcd}}\label{subsec:gcd}

\begin{proof}
By Proposition~\ref{prop:lambdas_finite} $\mathbb{P}\{\lambda_p(f_i(\mathcal{V}_1,\ldots,\mathcal{V}_s))<\infty\}=1$ for all $p\in\mathcal{P}$. By Lemma~\ref{lem:dimension} the dimension of a $\mathbb{Q}$-variety $A_{\mathbb{Q}}(f_1,\ldots,f_m)$ is at most $s-2$.  According to Proposition~\ref{prop:lang_weil}, for all $p$ large enough,
\begin{multline}\label{eq:gcd_finiteness}
\mathbb{P}\{\min_{i=1,\ldots, m} \lambda_p(f_i(\mathcal{V}_1,\ldots,\mathcal{V}_s))\geq 1\}=
\mathbb{P}\{\pi_0^{(p)}(f_i(\mathcal{V}_1,\ldots,\mathcal{V}_s))=0 \; \forall i=1,\ldots, m \}\\
=p^{-s}\#\{(x_1,\ldots,x_s)\in\mathbb{F}_p^s: f_i(x_1,\ldots,x_s)=0 \; \forall i=1,\ldots, m\}=O(p^{-2}),
\end{multline}
provided $s\geq 2$. For $s=1$, the probability vanishes for sufficiently large $p$. By the Borel-Cantelli lemma, this implies the a.s.~convergence of the series defining $\log G_{f_1,\ldots, f_m}$.

Fix $N\in\mathbb{N}$ and write, for $n\geq N$,
\begin{align*}
\log \GCD(f_1(\mathbf{U}_n^{(s)}),\ldots, f_m(\mathbf{U}_n^{(s)}))
&=
\left(\sum_{p\in\mathcal{P},p\leq N}+\sum_{p\in\mathcal{P},N < p \leq n}+\sum_{p\in\mathcal{P},p > n}\right)\log p \min_{i=1,\ldots, m} \lambda_p(f_i(\mathbf{U}_n^{(s)}))\\
&=:Y_1(n,N)+Y_2(n,N)+Y_3(n).
\end{align*}
By Corollary~\ref{cor:polys_p_adic_norms_converge} and the continuous mapping theorem
$$
Y_1(n,N)~\todistr~\sum_{p\in\mathcal{P},p\leq N}\log p \min_{i=1,\ldots, m} \lambda_p(f_i(\mathcal{V}_1,\ldots,\mathcal{V}_s)),\quad n\to\infty.
$$
As we have already shown, the right-hand side converges a.s.~to $\log G_{f_1,\ldots, f_m}$, as $N\to\infty$. Using Theorem 3.2 in~\cite{Billingsley} we see that it suffices to check that
\begin{align}
&\lim_{N\to\infty}\limsup_{n\to\infty} \mathbb{P}\{Y_2(n,N)\neq 0\}=0,\label{eq:bill1a}\\
&\lim_{n\to\infty} \mathbb{P}\{Y_3(n)\neq 0\}=0.\label{eq:bill1b}
\end{align}

The proof of~\eqref{eq:bill1b} is postponed to Proposition~\ref{prop:large_prime_common_divisor} in Section~\ref{sec:abscence}. Let us prove~\eqref{eq:bill1a}. For $p\in\mathcal{P}$ and $k=1,\ldots,s$, put $Z_{n,p}^{(k)}:=U_{n,k}(\mod p)$ and note that
\begin{align*}
&\mathbb{P}\{Y_2(n,N)\neq 0\}\\
&\leq \mathbb{P}\{\exists p\in\mathcal{P}: N<p\leq n, \lambda_p(f_i(\mathbf{U}^{(s)}_n))\geq 1 \; \forall i=1,\ldots, m\}\\
&\leq \sum_{p\in\mathcal{P},N<p\leq n}\mathbb{P}\{\lambda_p(f_i(\mathbf{U}^{(s)}_n))\geq 1 \; \forall i=1,\ldots, m\}\\
&=\sum_{p\in\mathcal{P},N<p\leq n}\mathbb{P}\{f_1(Z_{n,p}^{(1)},\ldots,Z_{n,p}^{(s)}) \equiv \ldots \equiv f_m(Z_{n,p}^{(1)},\ldots,Z_{n,p}^{(s)})\equiv 0(\mod p)\}\\
&\leq\sum_{p\in\mathcal{P},N<p\leq n}\left(\max_{j=0,\ldots,p-1}\mathbb{P}\{Z_{n,p}^{(1)}=j\}\right)^s \#\{(x_1,\ldots,x_s)\in\mathbb{F}_p^s:f_1(x_1,\ldots,x_s)= \cdots\\
&\hspace{10cm} = f_m(x_1,\ldots,x_s)=0\}.
\end{align*}
Note that, for $p\leq n$,
\begin{multline*}
\max_{j=0,\ldots,p-1}\mathbb{P}\{Z_{n,p}^{(1)}=j\}=\max_{j=0,\ldots,p-1}\mathbb{P}\{U_{n,1}\in \{j,j+p,\ldots,j+\lfloor (n-j)/p\rfloor p\}\\
\leq n^{-1}\max_{j=0,\ldots,p-1}(\lfloor (n-j)/p\rfloor +1)\leq \frac{1}{p}+\frac{1}{n}\leq \frac{2}{p}.
\end{multline*}
Thus, applying the Lang-Weil bound from Proposition~\ref{prop:lang_weil} we see that
$$
\lim_{N\to\infty}\limsup_{n\to\infty}\mathbb{P}\{Y_2(n,N)\neq 0\}
\leq  \lim_{N\to\infty}O\left(\sum_{p>N}p^{-2}\right)=0.
$$
This completes the proof of~\eqref{eq:bill1a} and of Theorem~\ref{theo:gcd}.
\end{proof}

\begin{rem}[Ekedahl-Poonen density formula]\label{eq:poonen}
Theorem~\ref{theo:gcd}, in particular, implies that the set
$$
\mathcal{R}:=\{(x_1,\ldots,x_s)\in\mathbb{Z}^s:\GCD(f_1(x_1,\ldots,x_s),\ldots,f_m(x_1,\ldots,x_s))=1\}
$$
possesses the asymptotic density, which is equal to
\begin{multline*}
\mathbb{P}\left\{\sum_{p\in\mathcal{P}}\log p \min_{i=1,\ldots, m} \lambda_p (f_i(\mathcal{V}_1,\ldots,\mathcal{V}_s))=0\right\}=\prod_{p\in\mathcal{P}}\mathbb{P}\{\min_{i=1,\ldots, m} \lambda_p (f_i(\mathcal{V}_1,\ldots,\mathcal{V}_s))=0\}\\
=\prod_{p\in\mathcal{P}}\left(1-\mathbb{P}\{\lambda_p (f_i(\mathcal{V}_1,\ldots,\mathcal{V}_s))\geq 1\;\forall i=1,\ldots,m\}\right)=\prod_{p\in\mathcal{P}}\left(1-\frac{s_p}{p^s}\right),
\end{multline*}
where $s_p:=\#\{(x_1,\ldots,x_s)\in\mathbb{F}_p^s:f_i(x_1,\ldots,x_s)\equiv 0(\mod p)\;\forall i=1,\ldots,m\}$. For the last passage we used the second equality in~\eqref{eq:gcd_finiteness}. This result is known in the literature as Ekedahl--Poonen formula, see~\cite{Bodin+Debes,Poonen}.
\end{rem}

\subsection{Proof of Corollary~\ref{theo:divisible}}\label{subsec:divisible}

We start by writing factorizations over $\mathbb{Q}$:
\begin{equation}\label{eq:theo_divisible_proof1}
f=c_f\prod_{i=1}^{L} h_i^{u_i}\quad\text{and}\quad g=c_g\prod_{i=1}^{L} h_i^{v_i},
\end{equation}
where $\{h_1,\ldots,h_L\}$ is the set of irreducible factors of $f$ and $g$ without multiplicities, $c_f,c_g\in\mathbb{Z}$ and $u_i,v_i\geq 0$. The assumption that $f$ does not divide $g$ implies $u_i>v_i$, for some $i=1,\ldots,L$. Clearly,
$$
\mathbb{P}\{f(\mathbf{U}_n^{(s)})\text{ divides } g(\mathbf{U}_n^{(s)})\} = \mathbb P\{\GCD (f(\mathbf{U}_n^{(s)}), g(\mathbf{U}_n^{(s)}))=f(\mathbf{U}_n^{(s)})\}.
$$
Let $\bar f \in \Z[x_1,\ldots, x_s]$ be a homogeneous polynomial of the same degree as $f$ obtained from $f$ by dropping all monomials except those having the highest degree $\deg f$. Recall that $\deg f \geq 1$, so that $\bar f$ is not constant. Then, by the continuous mapping theorem combined with Slutsky's lemma,
\begin{equation}\label{eq:f_of_U_n_converge}
\frac{f(\mathbf{U}_n^{(s)})}{n^{\deg f}}~\todistr~\bar f(\mathbf{U}^{(s)}_{\infty}),\quad n\to\infty.
\end{equation}
Therefore, it suffices to show that
\begin{equation}\label{eq:theo_divisible_proof2}
\frac{\GCD (f(\mathbf{U}_n^{(s)}), g(\mathbf{U}_n^{(s)}))}{n^{\deg f}}~\toprobab~0,\quad n\to\infty.
\end{equation}
Using~\eqref{eq:theo_divisible_proof1} and Lemma~\ref{lem:gcd} in the Appendix we obtain, for some $c_{f,g}\in\mathbb{Z}$,
\begin{multline}\label{eq:theo_divisible_proof3}
\GCD (f(\mathbf{U}_n^{(s)}), g(\mathbf{U}_n^{(s)}))\leq c_{f,g}\prod_{i,j=1}^{L}\GCD(h_i^{u_i}(\mathbf{U}_n^{(s)}),h^{v_j}_j(\mathbf{U}_n^{(s)}))\\
=c_{f,g}\prod_{i=1}^{L}\GCD(h_i^{u_i}(\mathbf{U}_n^{(s)}),h^{v_i}_i(\mathbf{U}_n^{(s)}))\prod_{i\neq j}\GCD(h_i^{u_i}(\mathbf{U}_n^{(s)}),h^{v_j}_j(\mathbf{U}_n^{(s)})).
\end{multline}
For every pair of indices $i\neq j$, by Theorem~\ref{theo:gcd} $\GCD(h_i^{u_i}(\mathbf{U}_n^{(s)}),h^{v_j}_j(\mathbf{U}_n^{(s)}))$ converges in distribution to an a.s.~finite random variable, since $h_i$ and $h_j$ do not have a common factor. Thus, the last product in~\eqref{eq:theo_divisible_proof3} is bounded in probability\footnote{It actually converges because $(\GCD(h_i^{u_i}(\mathbf{U}_n^{(s)}),h^{v_j}_j(\mathbf{U}_n^{(s)})))_{i\neq j}$ converge jointly as is readily seen from Remark~\ref{rem:joint_convergence_all_systems_of_polys}.}. Therefore,~\eqref{eq:theo_divisible_proof2} is a consequence of
\begin{equation}\label{eq:theo_divisible_proof4}
\frac{1}{n^{\deg f}}\prod_{i=1}^{L}\GCD(h_i^{u_i}(\mathbf{U}_n^{(s)}),h^{v_i}_i(\mathbf{U}_n^{(s)}))=\frac{1}{n^{\deg f}}\prod_{i=1}^{L}(h_i(\mathbf{U}_n^{(s)}))^{\min(u_i,v_i)}~\toprobab~0,\quad n\to\infty.
\end{equation}
It remains to note that the degree of the polynomial $\prod_{i=1}^{L}h_i^{\min(u_i,v_i)}$ is strictly smaller than $\deg f$ because $u_i>v_i$ for at least one $i=1,\ldots,L$. This immediately implies~\eqref{eq:theo_divisible_proof4}. The proof is complete.

\subsection{Proof of Theorem~\ref{theo:lcm}}\label{subsec:lcm}
As in the proof of Theorem~\ref{theo:gcd} we start by checking that the random series defining $L_{f_1,\ldots,f_m}$ converges a.s. By Proposition~\ref{prop:lambdas_finite} all summands in the definition of $L_{f_1,\ldots,f_m}$ are a.s.~finite. Let us show that the series converges a.s. To this end, note that for any set of nonnegative integers $a_1,\ldots,a_m\in\{0,1,2\ldots\}$ we have
\begin{equation}\label{eq:max_sum_relation}
\max_{i=1,\ldots,m}a_i\neq \sum_{i=1}^{m}a_i~\Longrightarrow~\exists i,j\in\{1,2,\ldots,m\},i\neq j:\; a_i\geq 1,a_j\geq 1.
\end{equation}
Thus, by the Borel-Cantelli lemma the series converges a.s. provided that
\begin{equation}\label{eq:lcm_proof1}
\sum_{p\in\mathcal{P}}\mathbb{P}\left\{\exists i,j\in\{1,2,\ldots,m\},i\neq j:\; \lambda_p(f_i(\mathcal{V}_1,\ldots,\mathcal{V}_s))\geq 1,\lambda_p(f_j(\mathcal{V}_1,\ldots,\mathcal{V}_s))\geq 1\right\}<\infty.
\end{equation}
Eq.~\eqref{eq:gcd_finiteness} implies
$$
\sum_{p\in\mathcal{P}}\mathbb{P}\left\{\lambda_p(f_i(\mathcal{V}_1,\ldots,\mathcal{V}_s))\geq 1,\lambda_p(f_j(\mathcal{V}_1,\ldots,\mathcal{V}_s))\geq 1\right\}<\infty,\quad i\neq j,
$$
where we used that $f_i$ and $f_j$ do not share a common factor of degree $>0$. Thus,~\eqref{eq:lcm_proof1} follows by the union bound.

In order to prove~\eqref{eq:lcm_convergence} we fix $N\in\mathbb{N}$ and decompose $\NLCM$ similarly as in the proof of Theorem~\ref{theo:gcd}
\begin{align*}
&\hspace{-1cm}\log \NLCM(f_1(\mathbf{U}_n^{(s)}),\ldots,f_m(\mathbf{U}_n^{(s)}))=\sum_{p\in\mathcal{P}}\log p \left(\max_{i=1,\ldots, m} \lambda_p (f_i(\mathbf{U}_n^{(s)}))-\sum_{i=1}^{m} \lambda_p (f_i(\mathbf{U}_n^{(s)}))\right)\\
&=\left(\sum_{p\in\mathcal{P},p\leq N}+\sum_{p\in\mathcal{P},N<p\leq n}+\sum_{p\in\mathcal{P},p>n}\right)\log p \left(\max_{i=1,\ldots, m} \lambda_p (f_i(\mathbf{U}_n^{(s)}))-\sum_{i=1}^{m} \lambda_p (f_i(\mathbf{U}_n^{(s)}))\right)\\
&=:\widetilde{Y}_1(n,N)+\widetilde{Y}_2(n,N)+\widetilde{Y}_3(n).
\end{align*}
By the continuous mapping theorem
\begin{equation}\label{eq:y_1_lcm_convergence}
\widetilde{Y}_1(n,N)~\todistr~\sum_{p\in\mathcal{P},p\leq N}\log p \left(\max_{i=1,\ldots, m} \lambda_p (f_i(\mathcal{V}_1,\ldots,\mathcal{V}_s))-\sum_{i=1}^{m} \lambda_p (f_i(\mathcal{V}_1,\ldots,\mathcal{V}_s))\right),\quad n\to\infty,
\end{equation}
and the right-hand side, in turn, converges a.s.~to $\log L_{f_1,\ldots,f_m}$, as $N\to\infty$. Using~\eqref{eq:max_sum_relation} and the union bound we obtain
\begin{multline}\label{eq:lcm_proof31}
\mathbb{P}\{\widetilde{Y}_3(n)\neq 0\}\leq \mathbb{P}\left\{\exists p\in\mathcal{P}: p>n, \max_{i=1,\ldots,m}\lambda_p (f_i(\mathbf{U}_n^{(s)}))\neq \sum_{i=1}^{m} \lambda_p (f_i(\mathbf{U}_n^{(s)})) \right\}\\
\leq \sum_{i,j=1,i\neq j}^{m}\mathbb{P}\left\{\exists p\in\mathcal{P}: p>n, \lambda_p (f_i(\mathbf{U}_n^{(s)}))\geq 1,\lambda_p (f_j(\mathbf{U}_n^{(s)}))\geq 1\right\}.
\end{multline}
The right-hand side converges to $0$, as $n\to\infty$, by Proposition~\ref{prop:large_prime_common_divisor}. By the union bound,
\begin{multline}\label{eq:lcm_proof32}
\mathbb{P}\{\widetilde{Y}_2(n,N)\neq 0\}\leq \sum_{i,j=1,i\neq j}^{m}\sum_{p\in\mathcal{P},N<p\leq n}\mathbb{P}\{\lambda_p(f_i(\mathbf{U}^{(s)}_n))\geq 1,\lambda_p(f_j(\mathbf{U}^{(s)}_n))\geq 1\}.
\end{multline}
Thus, repeating verbatim the proof of~\eqref{eq:bill1a}, we obtain $\lim_{N\to\infty}\limsup_{n\to\infty}\mathbb{P}\{\widetilde{Y}_2(n,N)\neq 0\}=0$. This finishes the proof of Eq.~\eqref{eq:lcm_convergence}.

To check~\eqref{eq:lcm_convergence2} we note that Corollaries~\ref{cor:poly} and~\ref{cor:polys_p_adic_norms_converge} actually imply a stronger version of~\eqref{eq:y_1_lcm_convergence}, namely
$$
\left(\widetilde{Y}_1(n,N),\frac{\mathbf{U}_n^{(s)}}{n}\right)~\todistr~\left(\sum_{p\in\mathcal{P},p\leq N}\log p \left(\max_{i=1,\ldots, m} \lambda_p (f_i(\mathcal{V}_1,\ldots,\mathcal{V}_s))-\sum_{i=1}^{m} \lambda_p (f_i(\mathcal{V}_1,\ldots,\mathcal{V}_s))\right),\mathbf{U}_{\infty}^{(s)}\right),
$$
for every fixed $N\in\mathbb{N}$. Thus, by~\eqref{eq:lcm_proof31} and~\eqref{eq:lcm_proof32},
$$
\left(\frac{\LCM(f_1(\mathbf{U}_n^{(s)}),\ldots,f_m(\mathbf{U}_n^{(s)}))}{\prod_{j=1}^{m}f_j(\mathbf{U}_n^{(s)})},\frac{\mathbf{U}_n^{(s)}}{n}\right)~\todistr~(L_{f_1,\ldots, f_m},\mathbf{U}_{\infty}^{(s)}), \quad n\to\infty.
$$
As in the proof of~\eqref{eq:f_of_U_n_converge}, the continuous mapping theorem and Slutsky's lemma imply
\begin{multline*}
\left(\frac{\LCM(f_1(\mathbf{U}_n^{(s)}),\ldots,f_m(\mathbf{U}_n^{(s)}))}{\prod_{j=1}^{m}f_j(\mathbf{U}_n^{(s)})},\frac{f_1(\mathbf{U}_n^{(s)})}{n^{\deg f_1}},\ldots,\frac{f_m(\mathbf{U}_n^{(s)})}{n^{\deg f_m}}\right)\\
\todistr~(L_{f_1,\ldots, f_m},\bar{f}_1(\mathbf{U}_{\infty}^{(s)}),\ldots,\bar{f}_m(\mathbf{U}_{\infty}^{(s)})), \quad n\to\infty,
\end{multline*}
which immediately yields~\eqref{eq:lcm_convergence2}.

\section{Absence of large common prime divisors}\label{sec:abscence}
\begin{prop}\label{prop:large_prime_common_divisor}
Let $f_1,\ldots, f_m \in \mathbb Z[x_1,\ldots, x_s]$ be $m\in \N$ non-zero polynomials that do not have a common factor of degree $>0$. Let $\mathbf{U}^{(s)}_n$ be uniformly distributed on $\{1,\ldots, n\}^s$. Then,
$$
\lim_{n\to\infty} \mathbb{P}\{ \exists p \in \mathcal P : p \geq n,  f_1(\mathbf{U}_n^{(s)}) \equiv \ldots \equiv  f_m(\mathbf{U}_n^{(s)}) \equiv 0 (\mod p)\} = 0.
$$
\end{prop}
For the proof we need several lemmas.
\begin{lemma}\label{lem:1}
For every number $M\in \mathbb N$ there is a number $B=B(M)$ depending only on $M$ such that the following holds for every $n\in \mathbb N$. Let $g_1,\ldots, g_m \in \mathbb Z[x]$ be polynomials in one variable such that:
\begin{enumerate}
\item[(a)] $m\leq M$ and $\deg g_i \leq M$ for all $i=1,\ldots, m$;
\item[(b)] the absolute values of all coefficients of $g_1,\ldots, g_m$ are bounded above by $M \cdot n^M$;
\item[(c)]  $g_1,\ldots, g_m$ do not have a common factor in $\mathbb Q[x]$ of degree $>0$.
\end{enumerate}
Then, we can find polynomials $a_1,\ldots, a_m\in \mathbb Z[x]$ and a number $A\in \mathbb N$ with $A\leq B n^B$  such that $a_1 g_1 + \ldots + a_m g_m = A$.
\end{lemma}
Let us note that since $\mathbb Q[x]$ is a principal ideal domain, there exist polynomials $b_1,\ldots, b_m \in \mathbb Q[x]$ with rational coefficients such that $b_1 g_1 + \ldots + b_m g_m = 1$. Multiplying these polynomials by a suitable number $A$ we can make their coefficients integer. Thus, the only nontrivial claim in the above lemma is the bound $A \leq B n^B$, which we claim to hold uniformly over all $g_1,\ldots, g_m$ and $n$ satisfying the above conditions. This uniformity will be crucial in what follows.

\begin{proof}
Essentially,  we apply the Euclidean algorithm while tracking the size of coefficients. We use induction over $\deg g_1 + \ldots +  \deg g_m$ (where we put $\deg 0:=0$).  If this number is $0$, then all polynomials $g_i$ are constant but not all of them are zero by Condition~(c). We can put $\alpha_i := 1$ if $g_i\geq 0$ and $\alpha_i:= -1$ if $g_i\leq 0$. Then,  $A = |g_1| + \ldots + |g_m|  \leq m Mn^M$, so that we can put $B:= M^2$.

Let now $\deg g_1 + \ldots + \deg g_m \geq 1$ and suppose that we proved the lemma for smaller values of this sum. Without loss of generality let $\deg g_1 \geq \max\{\deg g_2, \ldots, \deg g_m\}$. Then,  $\deg g_1 \geq 1$. By Condition~(c), some of the polynomials $g_2,\ldots, g_m$ is not identically zero. Let $g_2 \nequiv 0$. Write
$$
g_1(x) = c_px^p + \ldots + c_0,
\quad
g_2(x) = d_qx^q + \ldots + d_0,
\quad
c_i, d_j\in \Z,
\quad
p \geq q,
\quad
p \geq 1.
$$
Consider now instead of the tuple $(g_1,g_2,\ldots, g_m)$ the tuple $(d_q g_1 - c_q g_2 x^{p-q}, g_2,\ldots, g_m)$. Note that $\deg (d_q g_1 - c_q g_2 x^{p-q})< \deg g_1$. Also, the coefficients of the polynomials from the new tuple are integer and bounded above by $2 M^2 n^{2M}$, so , so that we can apply the induction assumption to the new tuple with $M$ replaced by $2M^2$. It follows that
$$
\tilde a_1(x) \cdot (d_q g_1(x) - c_q g_2(x) x^{p-q})  +  \tilde a_2(x) g_2(x) + \ldots + \tilde a_m(x) g_m(x) = A
$$
for suitable $\tilde a_1,\ldots, \tilde a_m \in \Z[x]$ and a number $A\in \N$, $A\leq B n^{B}$. After regrouping the terms this gives the claim.
\end{proof}

\begin{lemma}\label{lem:2}
For every number $M\in \mathbb N$ there is $C=C(M)$ depending only on $M$ such that the following holds for all $n\in \N$. Let $g_1,\ldots,g_m \in \mathbb Z[x]$ be polynomials satisfying Conditions (a), (b), (c) of Lemma~\ref{lem:1} and such that, additionally,
\begin{enumerate}
\item[(d)] There is no prime number $p\geq n$ dividing all coefficients of $g_1,\ldots, g_m$.
\end{enumerate}
Then, for the random variable $U_n$ uniformly distributed on $\{1,\ldots, n\}$ we have
$$
\mathbb{P}\{ \exists p \in \mathcal P : p \geq n,  g_1(U_n) \equiv \ldots \equiv  g_m(U_n) \equiv 0 (\mod p)\} \leq C/n.
$$
\end{lemma}
\begin{proof}
By Lemma~\ref{lem:1} we have $a_1 g_1 + \ldots + a_m g_m = A$ for some $A\in \mathbb N$  with $A \leq B n^B$ and some polynomials $a_1,\ldots, a_m \in \mathbb Z[x]$ with integer coefficients.
So, every common prime  divisor $p\geq n$ of $g_1(U_n), \ldots, g_m(U_n)$ must be a divisor of $A$. The number $A$ has at most $B+1$ distinct prime divisors $p_1,\ldots, p_\ell\geq n$, where we assumed that $n\geq B$. (For $n\leq B$ the claim is trivial since there are only finitely many choices for $g_1,\ldots, g_m$.)   So,
\begin{multline*}
\mathbb{P}\{\exists p \in \mathcal P : p \geq n,  g_1(U_n) \equiv \ldots \equiv  g_m(U_n) \equiv 0(\mod p)\}\\
\leq
\sum_{i=1}^\ell \mathbb{P} \{g_1(U_n) \equiv \ldots \equiv  g_m(U_n) \equiv 0(\mod p_i)\}.
\end{multline*}
Fix some $i\in \{1,\ldots, \ell\}$. Some of the coefficients of some polynomial $g_j$ is not divisible by $p_i$, by Condition~(d). So, the reduction of $g_j$ modulo $p_i$ is a non-zero polynomial. Thus, it has at most $\deg g_j \leq M$ zeros over $\mathbb F_{p_i}$. Since $p_i \geq n$ and hence all numbers $1,\ldots, n$ have different remainders modulo $p_i$, there are at most $M$ possible values of $U_n$ for which $g_j(U_n)$ is divisible by $p_i$. It follows that
$$
\mathbb P (g_1(U_n) \equiv \ldots \equiv  g_m(U_n) \equiv 0 \;\; \mod p_i) \leq M/n.
$$
The claim follows with $C:= (B+1)M$ since $\ell \leq B+1$.
\end{proof}

It is well known that the property of $2$ univariate polynomials to have a non-constant common divisor can be expressed as a polynomial condition on their coefficients. Given next is a generalization to any finite number of polynomials which is also a standard result in algebra, see~\cite{Resultant1,Resultant2}.
\begin{lemma}[Resultant]\label{lem:resultant}
Let $R$ be an integral domain, $m\in \N_0$. Fix ``degrees'' $d_1,\ldots, d_m\in \N_0$. There exist $L\in \N$ and polynomials $W_1,\ldots, W_L$ in $d_1 + \ldots + d_m + m$ variables (having integer coefficients) with the property that polynomials  $Q_1,\ldots, Q_m \in R[x]$ with $\deg Q_1 = d_1, \ldots, \deg Q_m = d_m$ have a nonconstant common divisor in $R[x]$ if and only if all polynomials $W_1,\ldots, W_L$, evaluated at the coefficients of $Q_1,\ldots, Q_m$, vanish.
\end{lemma}
\begin{proof}
For $m=2$ polynomials, we can take $L=1$ and $W_1$ to be the Sylvester resultant of $Q_1$ and $Q_2$.  For $m\geq 3$, we introduce new variables $u_2,\ldots, u_m$ and observe that $Q_1, \ldots, Q_m$ have a common factor in $R[x]$ if and only if $Q_1$ and $u_2 Q_2 + \ldots + u_m Q_m$ have a common factor in $R[x, u_1,\ldots, u_m] \equiv R'[x]$, where $R'= R[u_1,\ldots, u_m]$ is also an integral domain. The Sylvester resultant of $Q_1$ and $u_2 Q_2 + \ldots + u_m Q_m$, considered as elements of $R'[x]$, is a polynomial in the coefficients of $Q_1,\ldots, Q_m$ and the variables  $u_2,\ldots, u_m$. The resultant can be written as a sum of finitely many  monomials of the form $u_2^{\ell_2} \ldots u_m^{\ell_m}$ multiplied by certain polynomials in the coefficients of $Q_1,\ldots, Q_m$. Denote these polynomials (in some order)  by $W_1,\ldots, W_L$. Then, $W_1 = \ldots = W_L = 0$ if and only if the resultant of $Q_1$ and $u_2 Q_2 + \ldots + u_m Q_m$ vanishes, which is the case if and only if the polynomials $Q_1,\ldots, Q_m$ have a common factor.
\end{proof}
\begin{rem}\label{rem:resultant_one_direction}
If $\deg Q_1\leq d_1,\ldots, \deg Q_m \leq d_m$, then the ``only if'' direction of the above claim holds with the same proof: if $Q_1,\ldots, Q_m$ have a common factor, then $W_1,\ldots, W_L$, evaluated at the coefficients of $Q_1,\ldots, Q_m$, vanish.
\end{rem}

\begin{proof}[Proof of Proposition~\ref{prop:large_prime_common_divisor}]
We use induction over the number of variables $s$.  For $s=1$, the claim follows immediately from Lemma~\ref{lem:2}.

Take some $s\in \{2,3,\ldots\}$ and assume we proved the proposition for polynomials of $s-1$ variables. We prove it for polynomials with $s$ variables. The idea is to fix the numbers $x_1,\ldots, x_{s-1} \in \{1,\ldots, n\}$ and consider the polynomials $g_i (x_s) := f_i(x_1,\ldots, x_{s-1}, x_s)$ as  univariate polynomials in $x_s$. Clearly, $g_i \in \mathbb Z[x_s]$. For a sufficiently large $M\in \N$, Conditions~(a) and~(b) of Lemma~\ref{lem:1} are fulfilled. Let $C_n$, respectively  $D_n$, be the sets of all $(x_1,\ldots, x_{s-1}) \in \{1,\ldots, n\}^{s-1}$ for which $g_1,\ldots, g_m$ fail to satisfy  Condition~(c), respectively,  (d). Let $G_n$ be the complement of $C_n\cup D_n$, that is the set of all $(x_1,\ldots, x_{s-1}) \in \{1,\ldots, n\}^{s-1}$ for which both Conditions~(c) and~(d) are fulfilled.  Write $\Pi (x_1,\ldots, x_s) = (x_1,\ldots, x_{s-1})$ for the projection map removing the last coordinate.  Then,
\begin{align*}
\lefteqn{\mathbb{P}\{ \exists p \in \mathcal P : p \geq n,  f_1(\mathbf{U}_n^{(s)}) \equiv \ldots \equiv  f_m(\mathbf{U}_n^{(s)}) \equiv 0 \;\; \mod p, \Pi \mathbf{U}_n^{(s)} \in G_n)\}}\\
&=
\frac{1}{n^{s-1}} \sum_{(x_1,\ldots, x_{s-1}) \in G_n} \mathbb{P}\{\exists p \in \mathcal P : p \geq n,  f_1(x_1,\ldots, x_{s-1}, U_n) \equiv \ldots \\
&\hspace{8cm}\equiv  f_m(x_1,\ldots, x_{s-1}, U_n) \equiv 0(\mod p)\}\\
&\leq
\frac 1 {n^{s-1}} \sum_{(x_1,\ldots, x_{s-1}) \in G_n}
\frac{C}{n}
\leq
\frac{C}{n},
\end{align*}
where we applied Lemma~\ref{lem:2} to the polynomials $g_i (x_s) = f_i(x_1,\ldots, x_{s-1}, x_s)$. Note that the constant $C$ in Lemma~\ref{lem:2} does not depend on the choice of $x_1,\ldots, x_{s-1} \in \{1,\ldots, n\}$.

Let us check that $\mathbb{P}\{\Pi \mathbf{U}_n^{(s)} \in D_n\} \to 0$ as $n\to\infty$. Recall that $\Pi \mathbf{U}_n^{(s)} \in D_n$ means that all coefficients of the univariate polynomials $g_1(\Pi \mathbf{U}_n^{(s)}, x_s),\ldots, g_m(\Pi \mathbf{U}_n^{(s)}, x_s)$ have a common prime divisor $p\geq n$.  Consider the ring $R = \mathbb Z [x_1,\ldots, x_{s-1}]$. Then, we can view $h_i (x_s) := f_i(x_1,\ldots, x_{s-1}, x_s) \in R[x_s]$ as a polynomial in $x_s$ with coefficients in $R$. Let $q_1,\ldots, q_L\in R$ be the coefficients of the  polynomials $h_1(x_s), \ldots, h_m(x_s)$ listed in some order.  Then, $q_1,\ldots, q_L$ have no nonconstant common divisor in $R$ since otherwise $f_1,\ldots, f_m$ would have a nonconstant common divisor. We can then apply the induction assumption to $q_1,\ldots, q_L$ (which depend on $s-1$ variables and for which we assume Proposition~\ref{prop:large_prime_common_divisor} to hold). This yields
$$
\lim_{n\to\infty} \mathbb{P}\{\exists p \in \mathcal P : p \geq n,  q_1(\Pi \mathbf{U}_n^{(s)}) \equiv \ldots \equiv  q_L(\Pi \mathbf{U}_n^{(s)}) \equiv 0(\mod p)\} = 0.
$$
This proves that $\mathbb{P}\{\Pi \mathbf{U}_n^{(s)} \in D_n\} \to 0$ as $n\to\infty$.

Let us check that $\mathbb{P}\{\Pi \mathbf{U}_n^{(s)} \in C_n\} \to 0$, $n\to\infty$. We again consider $h_i (x_s)\in R[x_s]$ as polynomials in $x_s$ with coefficients in the integral domain $R = \mathbb Z [x_1,\ldots, x_{s-1}]$. By the hypothesis of Proposition~\ref{prop:large_prime_common_divisor}, these polynomials do not have a common factor in $R[x_s] = \Z [x_1,\ldots, x_s]$ of degree $>0$. By Lemma~\ref{lem:resultant} this implies that certain polynomial, say $W_1$, of their coefficients (which are elements in $R$) does not vanish in  $R$. Inserting in $W_1$ the coefficients  (which are polynomials in $x_1,\ldots, x_{s-1}$), we obtain certain \textit{non-zero} polynomial $W_2 \in \Z[x_1,\ldots, x_{s-1}]$. Now, $\Pi \mathbf U_n \in C_n$ means that the polynomials $g_1(\Pi \mathbf U_n, x_s),\ldots, g_m(\Pi \mathbf U_n, x_s)$, viewed as elements in $\Z[x_s]$, have a common non-constant factor, which, by Lemma~\ref{lem:resultant} and Remark~\ref{rem:resultant_one_direction}, implies that $W_2 (\Pi \mathbf U_n) = 0$. Since $W_2\nequiv 0$, we can apply Lemma~\ref{lem:schwartz_zippel} from the Appendix, which yields
$$
\mathbb{P}\{\Pi \mathbf{U}_n^{(s)} \in C_n\} \leq \mathbb{P}\{W_2 (\Pi \mathbf{U}_n^{(s)} = 0\} \leq \frac{\deg W_2}{n},
$$
which converges to $0$ as $n\to\infty$.
\end{proof}

\section*{Appendix}

\begin{lemma}\label{lem:gcd}
For $a_1,\ldots,a_n,b_1,\ldots,b_m\in\mathbb{N}$ we have
$$
\GCD(a_1\cdots a_n,b_1\cdots b_m)\leq \prod_{i=1}^{n}\prod_{j=1}^{m}\GCD(a_i,b_j).
$$
\end{lemma}
\begin{proof}
Using a crude bound
$$
\min(x_1+\cdots+x_n,y_1+\cdots+y_m)\leq \sum_{i=1}^{n}\sum_{j=1}^{m}\min (x_i,y_j),\quad x_i,y_j\geq 0,
$$
we obtain
\begin{multline*}
\GCD(a_1\cdots a_n,b_1\cdots b_m)=\prod_{p\in\mathcal{P}}p^{\min(\sum_{i=1}^{n}\lambda_p(a_i),\sum_{j=1}^{m}\lambda_p(b_j))}\leq \prod_{p\in\mathcal{P}}p^{\sum_{i=1}^{n}\sum_{j=1}^{m}\min(\lambda_p(a_i),\lambda_p(b_j))}\\
\leq \prod_{i=1}^{n}\prod_{j=1}^{m}\prod_{p\in\mathcal{P}}p^{\min(\lambda_p(a_i),\lambda_p(b_j))}=\prod_{i=1}^{n}\prod_{j=1}^{m}\GCD(a_i,b_j).
\end{multline*}
\end{proof}

\begin{lemma}\label{lem:haar_of_variety_is_zero}
Fix $p\in\mathcal{P}$. Let $f\in \mathbb{Q}_p[x_1,\ldots,x_s]$ be a non-zero polynomial over $p$-adic rationals and $\mu_p$ be the Haar measure on $\Z_p$. Then
$$
\mu_p^{\otimes s}
(\{x = (x_1,\ldots, x_s)\in \mathbb{Z}_p^s:f(x)=0\})=0.
$$
\end{lemma}
\begin{proof}
We use induction over $s$. For $s=1$, the polynomial $f$ has only finitely many zeros in $\Z_p$ since $f\nequiv 0$, hence the claim is true. Suppose the claim is true for polynomials of $s-1$ variables. Consider some non-zero polynomial $f\in \mathbb{Q}_p[x_1,\ldots,x_s]$. One of the variables (without loss of generality, $x_1$) appears in $f$ in degree $\geq 1$. Write $f(x_1,\ldots, x_s)= \sum_{j=0}^d x_1^j a_j(x_2,\ldots, x_s)$, where $d\geq 1$, $a_j \in \Q_p[x_2,\ldots, x_s]$ and $a_d \nequiv 0$. By induction hypothesis, the set $E\subset \Z_p^{s-1}$ consisting of the zeros of the polynomial $a_d(x_2,\ldots, x_s)$ is a $\mu_p^{\otimes (s-1)}$-zero set. Hence,
$$
\mu_p^{\otimes s}
(\{x\in \mathbb{Z}_p^s:f(x)=0, (x_2,\ldots, x_s) \in E\})=0.
$$
On the other hand, for every fixed  $(x_2,\ldots, x_s)\in \Z_p^{s-1} \backslash E$, the polynomial $x_1\mapsto f(x_1,\ldots, x_s)$ is non-zero and  has at most $d$ roots. By Fubini's theorem,
$$
\mu_p^{\otimes s}
(\{x\in \mathbb{Z}_p^s:f(x)=0, (x_2,\ldots, x_s) \notin E\})=0,
$$
and the proof is complete.
\end{proof}

\begin{prop}[The Schwartz-Zippel bound]\label{lem:schwartz_zippel}
Let $Q\in \mathbb Z[x_1,\ldots, x_s]$ be a non-zero polynomial and let $\mathbf{U}_n^{(s)} = (U_{n,1},\ldots,U_{n,s})$ be uniformly distributed on $\{1,\ldots, n\}^s$. Then,
$$
\mathbb{P}\{Q(\mathbf{U}_n^{(s)}) = 0\} \leq \frac{\deg Q}{n}.
$$
\end{prop}

\section*{Acknowledgments}

This work has been accomplished during AM's visit to Queen Mary University of London as Leverhulme Visiting Professor in July-December 2023. AM gratefully acknowledges financial support from the Leverhulme Trust. ZK has been supported by the German Research Foundation under Germany's Excellence Strategy  EXC 2044 -- 390685587, Mathematics M\"unster: Dynamics - Geometry - Structure.

\end{document}